\def\arxivversion{1} %
\ifx\arxivversion\undefined
    \documentclass[letterpaper, 10 pt, conference]{ieeeconf}  
\IEEEoverridecommandlockouts
\usepackage[sort,compress]{cite}
\usepackage[pdftex]{graphicx}
\usepackage{array}
\usepackage{url}
\usepackage{arydshln}
\usepackage{float}
\usepackage{scrextend}
\usepackage{blindtext}
\usepackage{amsmath,amsfonts,amssymb}
\usepackage{tikz}
\usetikzlibrary{positioning}
\usepackage{xspace}
\usepackage{mathrsfs,mathtools}
\usepackage{mathalfa}
\usepackage{euscript}
\usepackage{csquotes}
\usepackage{datetime}
\usepackage{arydshln}
\usepackage{multirow}
\usepackage{blindtext}
\usepackage[caption=false]{subfig}
\usepackage[font=footnotesize]{caption}
\usepackage{wrapfig}

\newcommand{\arxivver}[2]{%
  \ifx\arxivversion\undefined%
    #1%
  \else%
    #2%
  \fi%
}

\DeclareFontFamily{U}{txcal}{\skewchar \font =45}
\DeclareFontShape{U}{txcal}{m}{n}{<-> txr-cal}{}
\DeclareMathAlphabet{\mathcalpxtx}{U}{txcal}{m}{n}

\definecolor{mblue}{rgb}{0,0.4470,0.7410}
\definecolor{morange}{rgb}{0.8500,0.3250,0.0980}
\definecolor{myellow}{rgb}{0.9290,0.6940,0.1250}
\definecolor{mpurple}{rgb}{0.4940,0.1840,0.5560}
\definecolor{mgreen}{rgb}{0.4660,0.6740,0.1880}
\definecolor{mcyan}{rgb}{0.3010,0.7450,0.9330}
\definecolor{mred}{rgb}{0.6350,0.0780,0.1840}
\definecolor{mgreenblue}{rgb}{0.0,1.0,0.5}
\definecolor{parulablue}{rgb}{0.2431,0.1490,0.6588}
\definecolor{parulalblue}{RGB}{39,151,235}
\definecolor{parulagreen}{RGB}{129,204,89}
\definecolor{parulayellow}{RGB}{249,251,21}
\definecolor{verylightgray}{RGB}{248,248,250}

\definecolor{cblue}{rgb}{0,0.9,1}
\definecolor{corange}{rgb}{1,0.7,0}

\newcommand{\tss}[1]{\textsuperscript{#1}}

\newcounter{ass}

\newlength{\extendedwidth}
\arxivver{\setlength{\extendedwidth}{\linewidth}}{\setlength{\extendedwidth}{.7\linewidth}}%

\newcommand{\mc}[1]{\mathcal{#1}}

\newcommand{\mr}[1]{\mathrm{#1}}
\newcommand{\mb}[1]{\mathbb{#1}}

\newcommand{\msf}[1]{\mathsf{#1}}
\newcommand{\mt}[1]{\mathtt{#1}}
\newcommand{\mbf}[1]{\mathbf{#1}}

\newcommand{\dnx}{{n_\mr{x}}}
\newcommand{\dny}{{n_\mr{y}}}
\newcommand{\dnu}{{n_\mr{u}}}

\newcommand{\dnr}{{n_\mr{r}}}

\newcommand{\Nr}{{N_\mr{r}}}

\makeatletter
\DeclareRobustCommand\vdots{%
  \mathpalette\@vdots{}%
}
\newcommand*{\@vdots}[2]{%
  \sbox0{$#1\cdotp\cdotp\cdotp\m@th$}%
  \sbox2{$#1.\m@th$}%
  \vbox{%
    \dimen@=\wd0 %
    \advance\dimen@ -3\ht2 %
    \kern.5\dimen@
    \dimen@=\wd2 %
    \advance\dimen@ -\ht2 %
    \dimen2=\wd0 %
    \advance\dimen2 -\dimen@
    \vbox to \dimen2{%
      \offinterlineskip
      \copy2 \vfill\copy2 \vfill\copy2 %
    }%
  }%
}
\DeclareRobustCommand\ddots{%
  \mathinner{%
    \mathpalette\@ddots{}%
    \mkern\thinmuskip
  }%
}
\newcommand*{\@ddots}[2]{%
  \sbox0{$#1\cdotp\cdotp\cdotp\m@th$}%
  \sbox2{$#1.\m@th$}%
  \vbox{%
    \dimen@=\wd0 %
    \advance\dimen@ -3\ht2 %
    \kern.5\dimen@
    \dimen@=\wd2 %
    \advance\dimen@ -\ht2 %
    \dimen2=\wd0 %
    \advance\dimen2 -\dimen@
    \vbox to \dimen2{%
      \offinterlineskip
      \hbox{$#1\mathpunct{.}\m@th$}%
      \vfill
      \hbox{$#1\mathpunct{\kern\wd2}\mathpunct{.}\m@th$}%
      \vfill
      \hbox{$#1\mathpunct{\kern\wd2}\mathpunct{\kern\wd2}\mathpunct{.}\m@th$}%
    }%
  }%
}
\makeatother

\newcommand{\Acl}{A_\textsc{cl}}

\usepackage{amsthm}
\theoremstyle{definition}
\newtheorem{definition}{Definition}

\theoremstyle{plain}

\newtheorem{lemma}{Lemma}
\newtheorem{proposition}{Proposition}

\newtheorem{assumption}{Assumption}
\newtheorem{fact}{Fact}
\newtheorem{remark}{Remark}
\theoremstyle{remark}

\renewcommand{\qedsymbol}{$\blacksquare$}

\let\labelindent\relax
\usepackage[inline]{enumitem}

\newcommand{\comment}[1]{}
\newcommand{\remove}[1]{{\color{gray}#1}}

\newcommand{\citet}[1]{\cite{#1}}

\begin{document}
\title{Co-designing a Layered Control Architecture}%
\author{%
Chris Verhoek, Charis Stamouli, George J. Pappas and Nikolai Matni
\thanks{C.~Verhoek, G.~J.~Pappas, and N.~Matni are with the Dept. of Electrical and Systems Engineering, University of Pennsylvania, United States. C.~Verhoek is also with the Control Systems Group, Eindhoven University of Technology, The Netherlands. C.~Stamouli is with the Institut f{\"u}r Automatik, ETH Z{\"u}rich, Switzerland. Corresponding author: C.~Verhoek (\texttt{c.verhoek@tue.nl}).}%
\thanks{This work has been supported by AFOSR (FA9550261B038) and the European project COVER under the grant No.~101086228.}%
}

\maketitle
\thispagestyle{plain}
\pagestyle{plain}

\begin{abstract}
	Layered control architectures (LCAs) are ubiquitous in complex engineering systems. However, in practice they are largely designed ad hoc and mostly dictated by the experience of the designer. This is because there is currently no systematic and practically applicable theory available to holistically design LCAs. Hence, systematically reasoning about the optimality, stability, and feasibility of the LCA design is not possible. We set a first step toward such a theory by presenting a framework for co-designing a two-layer LCA for a linear time-invariant plant. The resulting co-design problem optimizes over \emph{architectural} design variables (e.g., reduced order model dimension and sampling ratio between the layers) and \emph{tuning}-based design variables (e.g., the numerical values of the feedback gain or the matrices of the reduced order model), all subject to constraints that ensure stability, feasibility, and compute properties of the closed-loop LCA. The proposed formulation allows to make informed tradeoffs along the Pareto front traced over different axes of optimality, e.g., speed, accuracy or computational effort. We demonstrate the effect of making various tradeoffs in a simple numerical example.

\end{abstract}

\section{Introduction}
Intricate designs of interconnected controllers operating at different levels of abstraction and on multiple timescales are at the core of the high-performance operation of complex, autonomous, cyber-physical systems, including robotics, aerospace, high-tech mechatronics, and power grids~\cite{matni2024quantitative}. While such \emph{layered control architectures}~(LCAs) are a near-universal design pattern, in practice they are still designed largely ad hoc: governed by the experience of the designers, and each layer designed and tuned in isolation. This often necessitates multiple (expensive) redesigns of the LCA to ensure system requirements are met, because interactions between the layers are often only verified a posteriori. The main reason for this design process is the absence of a systematic and theoretically founded theory of control architecture design, making it in fact impossible to make a priori formal statements about the optimality, stability, or feasibility of the LCA design. Such a control architecture design theory must enable the \emph{joint} design of layers in an LCA, i.e., a way to reason \emph{about} LCA design and co-design the layers to trade off different performance aspects of the LCA. {In this work, we aim to get a better understanding of the fundamental relations in LCAs by instantiating the co-design philosophy on a simple, linear, two-layer example, providing a tractable starting point. 
This results in a \emph{concrete} co-design problem, allowing us to make systematic tradeoffs in LCA design. %

\emph{Related work:\ }
We build on the ideas proposed in~\citet{matni2016theory, matni2024quantitative}, which consider joint and/or systematic design of LCAs. Complementary to these more abstract and system-theoretic approaches to LCA theory, several works design and certify a \emph{given} (two-layer) architecture in the context of %
multi-rate \emph{model predictive control}~(MPC)~\cite{rosolia2020multirate, csomay2022multirate, rosolia2022unified}, hierarchical MPC~\cite{scattolini2009architectures}, abstraction-based hierarchical control via (approximate) simulation relations~\cite{pappas2000hierarchically,girard2006hierarchical}, simulation functions~\cite{stamouli2025layered, girard2009hierarchical, stamouli2026layered,fainekos2009temporal}, reference governors~\cite{garone2017reference}, contract theory~\cite{incer2024layered, mazo2024contract, takayama2026safety}, and reachability-based methods~\cite{herbert2017fastrack}. 
All these works provide successful approaches to design the individual controllers of an a priori \emph{given} LCA by providing careful analysis of the interactions between the layers, and tune the resulting controllers based on these analyses. However, the important \emph{architectural choices} in the LCA (e.g, multi-rate sampling time ratio, reduced model order, control horizon length), which dictate the feasible design space, are left out of the design process.
The aforementioned approaches manage the models and constraints in the higher layers through satisfaction of exact consistency conditions (e.g., using abstractions and simulation relations), Lyapunov-based contracts, or barrier certificates. In our work, we \emph{explicitly} make such architectural choices part of the design process. 
 This is in line within the more abstract mathematical theory of co-design~\cite{censi2015mathematical, zardini2023thesis, zardini2021codesign}, which uses category theory to arrive at (Pareto) optimal designs of interconnected, heterogeneous systems through monotone maps. {In this work, we take the view point of the \emph{design meta-problem}~(DMP) in~\citet{verhoek2026quantitative}, which quantifies the tradeoffs in single/intra-layer control design. We combine this with the concepts in the MPC and co-design literature, which results in an explicit, quantifiable co-design problem.} This problem allows designers to systematically make tradeoffs in LCA design subject to both the architectural choices and the usual controller design parameters of the individual layers.

\emph{Contributions:\ } 
We take a first step towards a quantitative approach for the co-design of LCAs. By analyzing the interactions between two layers in a linear multi-rate LCA in terms of \emph{architecture}-based and \emph{tuning}-based design variables (denoted~$\phi$ and~$\theta$, resp.), we define a path toward tractable co-design of LCAs. This path provides explicit and computable conditions to determine whether a set of~$\phi$ and~$\theta$ variables results in a stable and feasible LCA design, i.e., it enables a priori certification and comparison between candidate architectures defined through~$\phi$ and~$\theta$. Concretely, we derive the closed-loop LCA error dynamics on both the timescales of the LCA, all in terms of~$\phi$ and~$\theta$. With these derivations, we formulate the constraints in our co-design problem: a $\phi,\theta$-dependent stability condition, ensuring a stable interconnection of the two layers, and an implicit, but computable set-valued map dependent on~$\phi,\theta$, with which we can conclude whether the LCA design is feasible and constraint satisfaction can be guaranteed. We prove this map is anti-tone, so iteration of the constraint sets yields monotonically tightening enclosures of the constraints for the planner. Putting these together allow us to cast the resulting co-design problem as a multi-objective optimization over the architecture-based and tuning-based design variables, providing a way to make explicit tradeoffs in LCA design. While efficiently solving the co-design problem remains still open, we demonstrate in a simple numerical example that \emph{we can search over the architectural choices}, intuitively illustrating the different behaviors of Pareto-optimal LCA designs.

\emph{Notation:\ }
$\mb{R}$, $\mb{N}$, and $\mb{N}_{\geq0}$ are the sets of real, natural and non-negative integer numbers, respectively. For a matrix $X\in\mb{R}^{n\times n}$, we denote the operator norm of~$X$ by~$\|X\|$, while $\rho(X)$ denotes its spectral radius. For sets $\mc{X,Y}\subset\mb{R}^n$, the Minkowski sum and Pontryagin difference of $\mc{X,Y}$ are $X\oplus Y := \{ x + y \mid x\in X, \ y\in Y\}$ and $X\ominus Y = \{ z \mid \{z\}\oplus Y \subseteq X\}$, respectively. 
All proofs and derivations can be found in the \arxivver{extended version of this paper.}{appendix.}

\section{Problem formulation}\label{s:problem}
\subsection{The plant and control architecture}\label{ss:plant+lca}
We design an LCA for the \emph{linear time-invariant}~(LTI), continuous-time plant:
\begin{equation}\label{eq:ctplant}
	\dot{x}(t) = A_\mr{c} x(t) + B_\mr{c} u(t),
\end{equation}
where $x(t)\in\mb{R}^\dnx$ and $u(t)\in\mb{R}^\dnu$. Throughout the design of the LCA, we will work with the discretized\footnote{While the discretization method can be arbitrary, the resulting discrete-time representation must be such that any controller stabilizing~\eqref{eq:plant} must stabilize~\eqref{eq:ctplant}. To avoid deferring from the main focus of the paper, we assume throughout exact discretization. This avoids this need to formulate additional stability conditions related to the interactions between the continuous-time plant and the discrete-time controller. This assumption can be lifted using the methods presented in~\citet{verhoek2026quantitative}, where we focus on this particular problem in a similar context.} representation of~\eqref{eq:ctplant}, sampled with sampling time~$\tau_\mr{L}$, which gives the discrete-time plant
\begin{equation}\label{eq:plant}
	x_{k+1} = A x_k + B u_k.
\end{equation}
The plant is subject to state and actuator constraints~$x_k\in\mc{X}\subset\mb{R}^\dnx$ and~$u_k\in\mc{U}\subset\mb{R}^\dnu$, respectively, where $\mc{X},\mc{U}$ are compact sets containing the origin. This plant is controlled by the two-layered hierarchical LCA, depicted in Fig.~\ref{fig:hierarchical}.
\arxivver{%
\setlength{\columnsep}{7pt}
\setlength{\intextsep}{7pt}
\begin{wrapfigure}{r}{0.51\columnwidth}
    \centering
    {\small
\begin{tikzpicture}[
    block/.style={rectangle, draw, minimum width=1.5cm, minimum height=1cm, align=center},
    arrow/.style={->, >=stealth, thick}
]

    \node[block] (P) {$\mathbf{P}$};
    \node[block, below=0.8cm of P] (C) {$\mathbf{C}$};

    \draw[arrow] ([xshift=0.5cm]P.south west) -- node[left] {$(\bar y,\bar r)$} ([xshift=0.5cm]C.north west);
    
    \draw[arrow] ([xshift=-0.5cm]C.north east) -- node[right] {$x$} ([xshift=-0.5cm]P.south east);

    \node[right=0.3cm of P] {Runs at~$\tfrac{1}{\tau_\mr{H}}$ Hz};
	\node[right=0.3cm of C] {Runs at~$\tfrac{1}{\tau_\mr{L}}$ Hz};

\end{tikzpicture}}
    \vspace{-6pt}
    \caption{Hierarchical control structure.}
    \label{fig:hierarchical}
\end{wrapfigure}
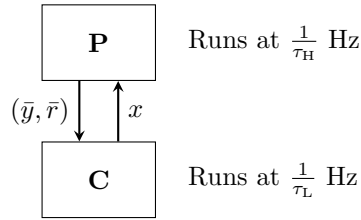
}{%
\begin{figure}[!ht]
    \centering
    {\small
\begin{tikzpicture}[
    block/.style={rectangle, draw, minimum width=1.5cm, minimum height=1cm, align=center},
    arrow/.style={->, >=stealth, thick}
]

    \node[block] (P) {$\mathbf{P}$};
    \node[block, below=0.8cm of P] (C) {$\mathbf{C}$};

    \draw[arrow] ([xshift=0.5cm]P.south west) -- node[left] {$(\bar y,\bar r)$} ([xshift=0.5cm]C.north west);
    
    \draw[arrow] ([xshift=-0.5cm]C.north east) -- node[right] {$x$} ([xshift=-0.5cm]P.south east);

    \node[right=0.3cm of P] {Runs at~$\tfrac{1}{\tau_\mr{H}}$ Hz};
	\node[right=0.3cm of C] {Runs at~$\tfrac{1}{\tau_\mr{L}}$ Hz};

\end{tikzpicture}}
    \caption{Hierarchical control structure.}
    \label{fig:hierarchical}
\end{figure}
}%
The `higher' layer is a planner~$\mbf{P}$, and the `lower' layer is a tracking controller~$\mbf{C}$, which directly controls the plant~\eqref{eq:plant}. The planner runs at~$1/\tau_\mr{H}$~Hz, while the tracking controller runs at~$1/\tau_\mr{L}$~Hz. Let~$\Nr:=\tau_\mr{H}/\tau_\mr{L}\in\mb{N}$ be the ratio between the sampling times. The planning layer~$\mbf{P}$ takes in the current state of~\eqref{eq:plant}, and computes a reference pair~$(\bar y,\bar r)$ for (a part of) the state and input, respectively, which is tracked by~$\mbf{C}$ for~$\Nr$ steps. After~$\tau_\mr{H}$ seconds, an updated reference pair is computed by $\mbf{P}$ and provided to the tracker.

Because of finite compute resources, the planning problem cannot be solved instantaneously. This implies that the reference computed by $\mbf{P}$ in the interval $[n,n+1]$ can only be tracked by $\mbf{C}$ in the interval $[n+1,n+2]$. The timeline sketched in Fig.~\ref{fig:timeline} aims to better illustrate the timescale separation of the layers in the considered LCA. 
\begin{figure}[!ht]
	\centering
	\includegraphics[width=\extendedwidth]{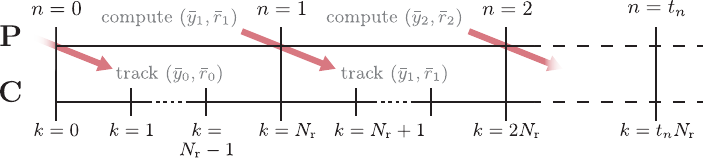}
	\caption{Timeline }\label{fig:timeline}
\end{figure}
The planner runs over the time indices~$n$, while the tracker runs over time indices~$k$. We can express the indices of the tracker,~$k$, in terms of the indices of the planner~$n$:~$k = n \Nr + j$,~$j=0,\dots, \Nr-1$. Throughout the paper, we use the short-hand notation~${z}_{j|n}$ to denote the value of a signal~$z$ on the fast timescale at the~$j$\tss{th} step in between time index~$n$ and~$n+1$. 

$\Nr$ is part of the architectural design variables~$\phi$, and it heavily influences the performance of the overall control scheme. In practice, the ratio~$\Nr$ is often taken between~$5$ and~$10$ as a rule of thumb~\cite{schmidt2020design, ellis2004control}, but mostly depends on the design. However, we do not know, and currently have no way of knowing, when and whether this is \emph{the best choice} for $\Nr$. Among others, this work aims to find the best~$\Nr$ based on the tradeoffs we are willing to make in co-designing the LCA. To make the co-design actionable, we fix the definitions of $\mbf{P}$ and $\mbf{C}$, which are in fact architectural design choices on their own. We will further comment on this decision in \S\ref{s:discussion}.

\subsection{The planning layer}
We take the planning layer as a receding-horizon predictive controller, which uses an approximate model of the plant~\eqref{eq:plant} in closed-loop with $\mbf{C}$ as predictor:
\begin{equation}\label{eq:planmod}
	y_{n+1} = A_\mr{p} y_n + B_\mr{p} r_n,
\end{equation}
with~$y_n\in\mc{Y}\subset\mb{R}^\dny$ and~$r_n\in\mc{R}\subset\mb{R}^\dnr$, where~$\dny\leq\dnx$ and~$\dnr\leq\dnu$ (typically~$\dny\lll\dnx$ and~$\dnr\ll\dnu$). We emphasize that the quality of~\eqref{eq:planmod} directly influences the performance of the overall LCA, and hence the measure of complexity that~\eqref{eq:planmod} can represent, directly governed by~$\dny,\dnr$, is an important architectural design choice. The information sharing between $\mbf{C}$ and $\mbf{P}$ is through the projection maps $\Pi_\mr{x}\in\mb{R}^{\dny\times\dnx}, \Pi_\mr{u}\in\mb{R}^{\dnr\times\dnu}$, such that~$y = \Pi_\mr{x}x$ and ~$r = \Pi_\mr{u}u$. In particular, we assume:
\begin{assumption}\label{ass:proj}
	The projection maps $\Pi_\mr{x}, \Pi_\mr{u}$ satisfy~$\Pi_\mr{x}\Pi_\mr{x}^\top=I_{\dny}$ and~$\Pi_\mr{u}\Pi_\mr{u}^\top=I_{\dnr}$, and~$(A,\Pi_\mr{x})$ is detectable.
\end{assumption}
Note that this common assumption {($\Pi_\mr{x}$ is often a selection matrix)} can be generalized with minor modifications by assuming~$\Pi_\bullet\Pi_\bullet^\top$ to be non-singular. For notational cleanliness, let $\dnr=\dnu$ and thus~$\Pi_\mr{u} = I_\dnu$.

With a slight misuse of notation, we use $\hat{y}_{i|n}$ to denote the~$i$\tss{th}-step-ahead prediction of $y$ at time index $n$ (still operating on the slow timescale, meaning that $\hat{y}_{1|n}\approx \hat{y}_{0|n+1}$). The receding-horizon nature of $\mbf{P}$ implies that for every time-step~$n$, the planner sets~$\hat{y}_{0|n} = \Pi_\mr{x}x_{n\Nr}$, and forward predicts length-$L$ trajectories of~\eqref{eq:planmod} to obtain the optimal reference for the tracker.\footnote{It is also possible to consider the case where the planner has only access to delayed, partial, and/or noisy information from the true system, cf.~\citet{matni2016theory}. We will consider these cases in follow-up works.} The optimization problem solved by the planner at time step~$n$ is of the form
\begin{subequations}\label{eq:planner}
\begin{align}
	{\arxivver{\textstyle}{}\min_{r_{i|n}}} \arxivver{\ }{\quad} & \ell_\mr{goal}\left(\{\hat y_{i|n}\}_{i=0}^{L-1}, \{r_{i|n}\}_{i=0}^{L-1}\right) \\ 
	\mr{s.t.}\quad & \hat y_{i+1|n} = A_\mr{p} \hat y_{i|n} + B_\mr{p} r_{i|n}, \ i \in [ 0, \, L-1], \\
	& \hat y_{i|n} \in \mc{Y}, \ r_{i|n}\in\mc{R}, \quad i \in [ 0, \, L-1], \\ & \hat y_{0|n} = \Pi_\mr{x}x_{n\Nr}, \label{eq:planner:stateupdate}
\end{align}
under cost~$\ell_\mr{goal}$, which we will further specify in \S\ref{s:codesign}.
\end{subequations}
Throughout, isolated properties of $\mbf{P}$ (recursive feasibility, convergence) are assumed to be handled by well-known standard design methods and are not the focus. We are interested in how the interaction with $\mbf{C}$ determines the architectural design choices of $\mbf{P}$ (e.g., length of~$L$, selection of~\eqref{eq:planmod}, or the relation between $\mc{X,U}$ and $\mc{Y,R}$). Finally, to ensure real-time implementation, the solution to the planning problem must be provided \emph{within}~$\tau_\mr{H}$ seconds.

Let the state and input reference tracked in the interval~$k=[n\Nr, (n+1)\Nr-1]$ be denoted by~$\bar{y}_n$ and~$\bar{r}_n$, respectively. Because of the aforementioned finite computational resources, there is a one-sample delay in the reference update that the planner provides to the tracker. The update relation of the reference signals is:
\begin{equation}\label{eq:plannerupdaterelation}
	(\bar{y}_n, \bar{r}_n) := \left(\hat y_{1|n-1},\ r_{0|n-1}\right).
\end{equation}

\subsection{The tracking layer}\label{ss:trackinglayer}
We set $\mbf{C}$ to be a static state-feedback controller, tracking the setpoint $\bar{y}_n$ and applying the feedforward~$\bar r_n$ from~$\mbf P$ for~$k=n\Nr, \dots, (n+1)\Nr-1$, giving the control law
\begin{equation}\label{eq:controllaw}
	u_k = K({\Pi_\mr{x}}x_k-\bar y_n) + \bar r_n.
\end{equation}
By letting~$\Acl:=A+BK{\Pi_\mr{x}}$, the tracker closed-loop is
\begin{equation}
	x_{k+1} = \Acl x_k - BK\bar y_n+B \bar r_n.
\end{equation}
Note that~$(\bar y_n,\bar r_n)$ remain constant for~$\Nr$ time steps and are updated at~$k=n\Nr$. The design of~$K$ is often based on a cost function $\ell_\mr{track}$, which we will further specify in \S\ref{s:codesign}, subject to the constraint that~$\Acl$ has only eigenvalues in the unit disc.\footnote{As will become clear later,~$\Acl$ \emph{may have unstable poles}, given that~$\Nr$ is finite. However, in practice, an unstable tracking layer is undesirable.}
As with the planning layer, to guarantee real-time application of the architecture, the computation of $u_k$ in~\eqref{eq:controllaw} must take less than~$\tau_\mr{L}$ seconds.
\begin{remark}[Shaping limitations of~$\Acl$]
	The conditions for having complete freedom in arbitrary pole placement of~$\Acl$ are~\cite[Thm.~3]{kimura1975pole}: $(A,B)$ must be controllable, $(A,\Pi_\mr{x})$ must be observable, and $\dny \geq \dnx - \dnu + 1$.
\end{remark}

\subsection{Problem statement}
The design of the two-layered LCA involves making several design choices that heavily affect the resulting closed-loop performance.
{While the design of the \emph{individual layers} is generally well understood, there is no systematic way to holistically quantify tradeoffs in LCA design, and jointly optimize over the design variables $\phi$ and $\theta$.} %
Through the presented simplified setting, we aim to take a first step towards closing this gap {by asking: \emph{Given the plant~\eqref{eq:ctplant} and $\mc{X},\mc{U}$, how to formulate a concrete, evaluable problem problem for co-designing $\mbf{P,C}$ subject to the design variables~$\phi$ and~$\theta$?}}

Toward that end, we carefully analyze the interactions between the layers and formulate conditions in~\S\ref{s:interactions} and~\S\ref{s:constraints} under which the LCA design is stable, feasible and satisfies the compute constraints. These conditions are the core ingredients for the {co-design} problem we formulate in \S\ref{s:codesign}. {This multi-objective optimization problem  \emph{i.}~enables the \emph{joint design} of $\mbf{P}$ and~$\mbf{C}$, \emph{ii.}~guarantees stability, feasibility, and compute constraints of the LCA, and \emph{iii.}~allows for making tradeoffs in the LCA design based on the Pareto-optimal solutions to the co-design problem.}

\section{Analyzing the interactions}\label{s:interactions}
 
Enabling systematic and joint design of the considered LCA starts with analyzing the interactions between the layers and the overall performance. We analyze the tracking error in the interval~$k \in [n\Nr, (n+1)\Nr-1]$, which we denote\footnote{Explicit derivation of the projected error term~$\bar e_{i|n} = \Pi_\mr{x}x_{n\Nr + i} - \bar y_n$ (which is the one of interest) is only possible under the harsh assumption that the tracked subspace is $\Pi$-invariant, i.e., $\exists \mc{A} : \Pi_\mr{x}(A+BK{\Pi_\mr{x}}) = \mc{A}\Pi_\mr{x}$ (or $\Pi_\mr{x}^{-1}$ exists, which is generally not true because $\dny\ll\dnx$). Therefore, we analyse the error with a lifted $\bar y_n$ and later project back to~$\bar e_{i|n}$, if needed.} by~$e_{i|n} = x_{n\Nr + i} - \Pi_\mr{x}^\top\bar y_n$ for $i \in[0, \Nr-1]$. Note that $e_{\Nr|n}=e_{0|n+1}$. 
\subsection{Error evolution on the fast timescale}
It is straightforward to characterize the evolution of $e_{i|n}$ for $i \in[0, \Nr-1]$:
\begin{subequations}\label{eq:errorevolution-n-n+1}
\begin{align}
	e_{i+1|n}  & = \Acl x_{n\Nr + i} - BK{\Pi_\mr{x}\Pi_\mr{x}^\top}\bar y_n + B\bar r_n - {\Pi_\mr{x}^\top}\bar y_n \\
	& = \Acl e_{i|n} + (A - I){\Pi_\mr{x}^\top} \bar y_n + B\bar r_n, 
\end{align}
thus, by defining $g_n:=(A - I){\Pi_\mr{x}^\top} \bar y_n + B\bar r_n$,
\begin{equation}\label{eq:errorevolution}
	e_{i|n} = \Acl^i e_{0|n} + \big(\textstyle\sum_{j=0}^{i-1}\Acl^j\big)g_n, \quad i \in [0, \Nr-1]. 
\end{equation}
\end{subequations}
\begin{remark}[Interpretation of $g_n$]\label{rem:interpretationgn}
	The tracking controller~$\mbf{C}$ is designed such that it tracks the reference $(\bar{y}_n,\bar{r}_n)$. By the definition of~$\mbf{C}$, it is inherently assumed that $(\Pi_\mr{x}^\top\bar{y}_n,\bar{r}_n)$ is an equilibrium point of~\eqref{eq:plant}, i.e.,
	 \[ \Pi_\mr{x}^\top\bar{y}_n = A \Pi_\mr{x}^\top\bar{y}_n + B\bar{r}_n \Leftrightarrow 0 = {(A -I)\Pi_\mr{x}^\top\bar{y}_n + B\bar{r}_n=g_n}. \]
	 Hence, $g_n$ represents the error introduced in the closed-loop due to the deviation of $(\Pi_\mr{x}^\top\bar{y}_n,\bar{r}_n)$ \emph{not} being an equilibrium point. Later, we will see that proper co-design of the architecture will allow us to \emph{counteract this accumulating error by design}, as it can use the accumulating error in anticipation for the reference update at $n+1$.
\end{remark}
\subsection{Error evolution on the slow timescale}
Analyzing how the error~$e_{i|n}$ evolves between~$n$ and~${n+1}$ will show the effect of an updated reference~$(\bar y_n, \bar r_n)$ from~$\mbf{P}$. Let for convenience~$i=0$, and define~$S_\Nr := \sum_{i=0}^{\Nr-1}\Acl^i$.\footnote{Note that $\sum_{i=0}^n X^i= (I-X)^{-1}(I - X^{n+1})$ and $\lim_{n\to\infty} \sum_{i=0}^n X^i = (I-X)^{-1}$ for $\rho(X)<1$.} 
The error evolution on the slow timescale is:
\begin{subequations}\label{eq:errorevolve:final}
\begin{equation}\label{eq:errorevolve:simple}
	e_{0|n+1} = \Acl^{\Nr} e_{0|n} + S_\Nr g_n + {\Pi_\mr{x}^\top}(\bar{y}_n - \bar{y}_{n+1}),
\end{equation}
see \arxivver{the extended version of the paper}{Appendix~\ref{app:deriveslowerror}} for its derivation. The form~\eqref{eq:errorevolve:simple} hides a dynamic relationship in~$\bar{y}_{n+1}$ due to its implicit dependence on~$e_{0|n}$. Therefore, by substituting $ x_{n\Nr} = e_{0|n} + \Pi_\mr{x}^\top\bar{y}_n$\arxivver{,}{ in the final step of Appendix~\ref{app:deriveslowerror},} we alternatively write~\eqref{eq:errorevolve:simple} as
\begin{multline}\label{eq:errorevolve_alt}
	e_{0|n+1} = \big[\Acl^{\Nr} - \Pi_\mr{x}^\top A_\mr{p} \Pi_\mr{x}\big]e_{0|n} \\ + \big[S_\Nr(A-I)\Pi_\mr{x}^\top - \Pi_\mr{x}^\top(A_\mr{p}-I)\big]\bar{y}_n \\+ \big[S_\Nr B - \Pi_\mr{x}^\top B_\mr{p}\big] \bar{r}_n + \Pi_\mr{x}^\top B_\mr{p} (\bar{r}_n- \bar{r}_{n+1}).
\end{multline}
\end{subequations}
\begin{remark}[Interpretations of~\eqref{eq:errorevolve:final}]
	We make a few insightful observations from both forms in~\eqref{eq:errorevolve:final}:
	\begin{itemize}[leftmargin=*, nosep]
		\item Closed-loop stability of the LCA is \emph{not} guaranteed by individually stable~$\mbf{C}$ and~$\mbf{P}$: For $\rho(\Acl)<1$, $\rho(\Acl^i)<1$ for every~$i$, but a combination of an unfortunately designed planner model $A_\mr{p}$ and projection map $\Pi_\mr{x}$ can still result in $\rho(\Acl^{\Nr} - \Pi_\mr{x}^\top A_\mr{p} \Pi_\mr{x})\geq 1$, i.e., an \emph{unstable} LCA.
		\item \eqref{eq:errorevolve_alt} dictates the modeling requirements for~\eqref{eq:planmod}: $\Pi_\mr{x}$ and $A_\mr{p}$ must be such that $\Pi_\mr{x}^\top A_\mr{p}\Pi_\mr{x}$ models the $\Nr$-step-ahead closed-loop dynamics~$\Acl$ of $\mbf{C}$. 
		\item In~\eqref{eq:errorevolve:simple},~$g_n$ is multiplied by the finite-horizon step response matrix $S_\Nr$, representing the compounding impact of the mismatch discussed in Remark~\ref{rem:interpretationgn}. The term~${\Pi_\mr{x}^\top}(\bar{y}_n - \bar{y}_{n+1})$ can \emph{counteract} this compounding mismatch if $\mbf{P}$ is designed such that the updated reference is \emph{exactly} canceling~$S_\Nr g_n$. This is achieved by modeling~\eqref{eq:planmod} such that the second and third term in~\eqref{eq:errorevolve_alt} are zero.
		\item The first three terms in~\eqref{eq:errorevolve_alt} show a clear interaction between the designs of $\mbf P$, $\mbf C$, and the plant~\eqref{eq:plant}, and dictate the LCA design choices. The last term can be seen as an external `impulse error' the system gets when the reference is updated. Our framework allows for making tradeoffs between focusing on the model quality and reducing the magnitude of this impulse error.
	\end{itemize}
\end{remark}
To reduce notational clutter, we use a short form of~\eqref{eq:errorevolve_alt} in the remainder of the paper by introducing
\begin{gather*} 
	\!\mc{A}_\mbf{e}:= \Acl^{\Nr} - \Pi_\mr{x}^\top A_\mr{p} \Pi_\mr{x}, \hspace{1.1cm}  \mc{B}^\mr{r}_\mbf{e}:= S_\Nr B- \Pi_\mr{x}^\top B_\mr{p}, \\
	\mc{B}^\mr{y}_\mbf{e}:= S_\Nr (A - I)\Pi_\mr{x}^\top - \Pi_\mr{x}^\top(A_\mr{p}-I), \quad  \mc{B}^\mr{d}_\mbf{e} :=\Pi_\mr{x}^\top B_\mr{p},
\end{gather*}
such that~\eqref{eq:errorevolve_alt} is rewritten as
\begin{equation}\label{eq:errorevolve_alt:short}
	e_{0|n+1} = \mc{A}_\mbf{e} e_{0|n} + \mc{B}^\mr{y}_\mbf{e} \bar{y}_n + \mc{B}^\mr{r}_\mbf{e} \bar{r}_n + \mc{B}^\mr{d}_\mbf{e} ( \bar{r}_n - \bar{r}_{n+1}).
\end{equation}

\begin{remark}[Connection to approximate simulation]
Using $S_{\Nr}(\Acl-I)=\Acl^{\Nr}-I$, one finds $\mc{B}^\mr{y}_\mbf{e}=0$ if and only if $\Pi_\mr{x}^\top A_\mr{p}=\Acl^{\Nr}\Pi_\mr{x}^\top-S_{\Nr}BK$. This is the consistency condition $PA'=AP+BQ$
of~\citet{girard2009hierarchical} for the $\Nr$-step lifted system, with
$P=\Pi_\mathrm{x}^\top$; likewise, Assumption~\ref{ass:proj} corresponds to $C'=CP$. Our framework does not require the exact satisfaction of the consistency condition in~\citet{girard2009hierarchical}: its violation is absorbed in the constraint-tightening map we present in~\S\ref{s:constraints} (through~$\bar{a}(\mathcal{Y})$ in~\eqref{eq:mcMDelta}).
\end{remark}

\section{Constraining the planning layer}\label{s:constraints}
The closed-loop LCA must ensure constraint satisfaction, while the tracker $\mbf C$ does not enforce $x_k\in\mc{X},u_k\in\mc{U}$. Hence, the planner $\mbf P$ must be designed such that for any provided reference (update), the tracking closed-loop is guaranteed to satisfy the constraints. In this section, we use~\eqref{eq:errorevolve_alt:short} {and standard arguments from the MPC literature to construct a set-valued map} that computes $\mc{Y}, \mc{R}$ as a function of the design variables of~$\mbf P$ and~$\mbf{C}$. {The antitone properties of this map, proven in~\S\ref{ss:prop_of_F}, makes the co-design formulation in~\S\ref{s:codesign} tractable and allows us} to evaluate whether a set of design choices realizes a feasible LCA.
\subsection{Construction of $\mc{Y}, \mc{R}$}\label{ss:construction-of-Y-R}
Define the constants
\begingroup\allowdisplaybreaks
\begin{align*}
	 \bar{a}(\mc{Y})  &:= {\textstyle \sup_{\bar{y}_n\in\mc{Y}}\|\mc{B}^\mr{y}_\mbf{e} \bar{y}_n\|, \quad 
	 \bar{b}(\mc{R})  := \sup_{\bar{r}_n\in\mc{R}}\|\mc{B}^\mr{r}_\mbf{e} \bar{r}_n\|},\\
	 \bar{d}(\mc{R})  &:= {\textstyle\sup_{\bar{r}_n,\bar{r}_{n+1}\in\mc{R}}\|\mc{B}^\mr{d}_\mbf{e} (\bar{r}_n -\bar{r}_{n+1})\|},\\
	 \bar{g}(\mc{Y},\mc{R})  &:= {\textstyle\sup_{\bar{y}_n\in\mc{Y}, \bar{r}_n\in\mc{R}}\|(A - I){\Pi_\mr{x}^\top} \bar y_n + B\bar r_n\|},
\end{align*}
\endgroup
which in turn define the sets 
\begin{subequations}\label{eq:mcMDelta}
\begin{align}
	\mc{W}(\mc{Y},\mc{R})& :=\{\,\varepsilon_\mr{a}\in\mb{R}^{\dnx} : \|\varepsilon_\mr{a}\| \leq \bar{a}(\mc{Y}) \, \}\notag \\ 
    & \qquad \oplus \ \{\,\varepsilon_\mr{b}\in\mb{R}^{\dnx} : \|\varepsilon_\mr{b}\| \leq \bar{b}(\mc{R})\,\} \notag\\ 
    & \qquad\quad \oplus \ \{\,\varepsilon_\mr{d}\in\mb{R}^{\dnx} : \|\varepsilon_\mr{d}\| \leq \bar{d}(\mc{R})\,\}, \\
	\mc{G}(\mc{Y},\mc{R})& :=\{\,\varepsilon_\mr{g}\in\mb{R}^{\dnx} : \|\varepsilon_\mr{g}\| \leq \bar{g}(\mc{Y},\mc{R}) \, \}.
\end{align}
\end{subequations}
Then, we can use $\mc{W}$ to construct the \emph{robust positive-invariant}~(RPI) set~\cite{rakovic2005invariant, kolmanovsky1998theory} of the planning dynamics on the slow timescale:
\begin{equation}
	\mc{E}_0(\mc{Y},\mc{R}) := {\textstyle\bigoplus_{t=0}^\infty}\left(\Acl^{\Nr} - \Pi_\mr{x}^\top A_\mr{p} \Pi_\mr{x}\right)^t\mc{W}(\mc{Y},\mc{R}),
\end{equation}
which states that if $e_{0|0}\in\mc{E}_0$, then $e_{0|n}\in\mc{E}_0$ for $n\in\mb{N}_{\geq1}$. To ensure that the error, and thus the state, remains bounded \emph{between} the reference updates, i.e., $e_{j|n}$, for $j\in[1,\Nr-1]$, we use the set $\mc{G}(\mc{Y},\mc{R})$ to construct an \emph{intra-period} tube, which is defined in the fast timescale. The `slices' of this tube are given by
\begin{equation}
	\mc{E}_j(\mc{Y},\mc{R}) := \Acl^j\mc{E}_0(\mc{Y},\mc{R})\oplus S_j\mc{G}(\mc{Y},\mc{R}).
\end{equation}
Hence, the intra-period tube can be constructed as:
\begin{equation}\label{eq:mcEstar}
	\mc{E}^\star(\mc{Y},\mc{R}):= {\textstyle\bigcup_{j=0}^{\Nr-1}}\mc{E}_j(\mc{Y},\mc{R}).
\end{equation}
Let $\mc{P}:=\mr{range}(\Pi_\mr{x}^\top)$ denote the column space of $\Pi_\mr{x}^\top$. Now we can define the set-valued function that returns the erosion of $(\mc{X},\mc{U})$ for a given $(\mc{Y},\mc{R})$:
\begin{equation}\label{eq:fpit-constraints}
	\msf{F}(\mc{Y},\mc{R}):=\big(\Pi_\mr{x}(\mc{X}\ominus\mc{E}^\star(\mc{Y},\mc{R})\cap\mc{P}), \ \mc{U}\ominus K \Pi_\mr{x}\mc{E}^\star(\mc{Y},\mc{R})\big).
\end{equation}
The constraint sets $(\mc{Y},\mc{R})$ for the planner $\mbf{P}$ that will guarantee constraint satisfaction of the LCA are the solution to the fixed point problem 
\begin{equation}\label{eq:fixedpointYR}
	\msf{F}(\mc{Y}_\star,\mc{R}_\star) = (\mc{Y}_\star,\mc{R}_\star).
\end{equation}

\subsection{Properties and computation of~$\msf{F}$}
An evident question in formulating a tractable co-design problem is: \emph{Under which conditions does there exist a solution $(\mc{Y}_\star,\mc{R}_\star)$ to~\eqref{eq:fixedpointYR} and is $(\mc{Y}_\star,\mc{R}_\star)$ even computable?}

\subsubsection{Existence of $(\mc{Y}_\star,\mc{R}_\star)$}\label{ss:prop_of_F} 
The properties of the operations for the construction of~$\msf F$ allow for the formulation of a condition under which $\msf F$ is anti-tone %
w.r.t. its inputs:
\begin{lemma}\label{lem:antitoneF}
	If $\rho(\Acl^{\Nr} - \Pi_\mr{x}^\top A_\mr{p} \Pi_\mr{x})<1$, then $\msf{F}$ is anti-tone in $(\mc{Y},\mc{R})$.
\end{lemma}
From Lemma~\ref{lem:antitoneF} 
we can directly conclude that $\msf{F}\circ\msf{F}$ is \emph{monotone} in $(\mc{Y},\mc{R})$. Hence, by the Knaster-Tarski theorem~\cite{davey2002introduction}, $\msf{F}\circ\msf{F}$ has a fixed point\footnote{Due to space restrictions, we omit the details on formally defining a complete lattice on which $\msf{F}$ acts.} under $\rho(\Acl^{\Nr} - \Pi_\mr{x}^\top A_\mr{p} \Pi_\mr{x})<1$: the first step in making $\msf{F}$ computable. From the properties of~$\msf{F}$ we  conclude:
\begin{proposition}
	If $\rho(\Acl^{\Nr} - \Pi_\mr{x}^\top A_\mr{p} \Pi_\mr{x})<1$, then either
	\begin{enumerate}[label={\roman*.}]
		\item There exists sets~$(\mc{Y}_\star,\mc{R}_\star)$ such that~\eqref{eq:fixedpointYR} holds,
		\item There exists sets~$(\mc{Y}^1_\star,\mc{R}^1_\star)$,~$(\mc{Y}^2_\star,\mc{R}^2_\star)$ such that $\msf{F}(\mc{Y}^1_\star,\mc{R}^1_\star) = (\mc{Y}^2_\star,\mc{R}^2_\star)$ and $\msf{F}(\mc{Y}^2_\star,\mc{R}^2_\star) = (\mc{Y}^1_\star,\mc{R}^1_\star)$.
	\end{enumerate}
\end{proposition}
Hence, $\msf{F}$ has either a fixed point $(\mc{Y}_\star,\mc{R}_\star)$ or a \emph{two-cycle pair}~$(\mc{Y}^1_\star,\mc{R}^1_\star)$,~$(\mc{Y}^2_\star,\mc{R}^2_\star)$. Note, however, that we do not \emph{need} a fixed point to ensure constraint satisfaction in the LCA. We simply need that the constraint sets of $\mbf{P}$ satisfy $ (\mc{Y},\mc{R}) \subseteq(\mc{Y}_\star,\mc{R}_\star)$, preferably with $(\mc{Y}_\star,\mc{R}_\star)\setminus(\mc{Y},\mc{R})$ as small as possible. Next, we discuss how to achieve this in the computation of $(\mc{Y},\mc{R})$. 
\subsubsection{Computing $(\mc{Y},\mc{R})$}
The respective anti-tone and monotone property of $\msf{F}$ and $\msf{F}\circ\msf{F}$ makes it possible to compute $(\mc{Y},\mc{R})$ using a standard fixed-point iteration scheme. To show convergence of this scheme, we first introduce some short-hand notation.
Let $\mc{S}_k:=(\mc{Y}_k,\mc{R}_k)$, $\overline{\mc{S}}_k:= \mc{S}_{2k}$, $\underline{\mc{S}}_k:= \mc{S}_{2k+1}$ and, with slight abuse of notation, $\mc{S}_{k+1}:=\msf{F}(\mc{S}_{k})$. Additionally, define $\mc{D}_k:=d_H(\underline{\mc{S}}_k, \overline{\mc{S}}_k)$ as the Hausdorff distance\arxivver{}{\footnote{See Appendix~\ref{app:hausdorff} for a formal definition of the Hausdorff distance.}} between the sets $\underline{\mc{S}}_k$ and $\overline{\mc{S}}_k$, i.e., the distance between the sets $\msf{F}(\mc{Y}_{2k},\mc{R}_{2k})$ and $(\mc{Y}_{2k},\mc{R}_{2k})$. Then we have: %
\begin{lemma}\label{lem:disFcontr}
	Let $\rho(\Acl^{\Nr} - \Pi_\mr{x}^\top A_\mr{p} \Pi_\mr{x})<1$ and set $\mc{S}_0=(\Pi_\mr{x}\mc{X}, \, \mc{U})$. Then, for all $k\geq 0$,
	\( \underline{\mc{S}}_k\subseteq\underline{\mc{S}}_{k+1}\subseteq\overline{\mc{S}}_{k+1}\subseteq\overline{\mc{S}}_k\)
	and $\mc{D}_{k}\geq \mc{D}_{k+1}$.
\end{lemma}
Lemma~\ref{lem:disFcontr} allows us to conclude two important facts:
\begin{enumerate}[leftmargin=*, nosep, label={\emph{\roman*.}}]
	\item It is easy to show that $\underline{\mc{S}}_k\subseteq\mc{S}_\star^{\bullet}\subseteq\overline{\mc{S}}_k$ for any $k$. Hence, $[\underline{\mc{S}}_k,\,\overline{\mc{S}}_k]$ is a certified enclosure of every fixed point of $\msf{F}\circ\msf{F}$, tightening monotonically as~$k$ grows, whether or not it ever collapses to a single set~$\mc{S}_\star$.
	\item $\mc{D}_k$ provides a measure for $\mc{S}_\star \setminus (\mc{Y},\mc{R})$: the smaller~$\mc{D}_k$, the closer we are to the fixed point/two-cycle pair of $\msf{F}$.
\end{enumerate}
Hence %
for a given set of design variables ($\Nr, \dny,A_\mr{p},K$, etc.), we can recursively compute~$\msf{F}$ until $\mc{D}_{k}-\mc{D}_{k+1}$ is less than a certain threshold, and set $(\mc{Y},\mc{R})\gets\underline{\mc{S}}_{k+1}$.%
\arxivver{}{\begin{remark}[On computing $\mc{Y},\mc{R}$]
	For general compact input sets $\mc{S}_k$ containing the origin, computing $F(\mc{S}_k)$ is relatively straightforward. There are standard toolboxes available, e.g.~\cite{mpt3}, to compute Minkowski additions, Pontryagin differences and unions of sets. Furthermore, much research is spent %
	on computing the RPI set~$\mc{E}_0(\mc{Y},\mc{R})$, see, e.g.,~\citet{kolmanovsky1998theory, rakovic2005invariant, mayne2014model}. Computation can be sped up by taking certain assumptions on $\mc{X,U,Y,R}$. %
	Assuming $\mc{X,U}$ as scaled unit balls around the origin, and parametrizing $(\mc{Y},\mc{R})$ as $(\gamma_\mr{y}\mb{B}_\mc{Y},\gamma_\mr{r}\mb{B}_\mc{R})$ gives the following program to obtain the constraint sets for $\mbf{P}$:
	\begin{align*}
		{\textstyle\max_{\gamma_\mr{y},\gamma_\mr{r}\geq 0}} \quad &  \mr{volume}(\gamma_\mr{y}\mb{B}_\mc{Y},\gamma_\mr{r}\mb{B}_\mc{R}) \\ \mr{s.t.} \quad& (\gamma_\mr{y}\mb{B}_\mc{Y},\gamma_\mr{r}\mb{B}_\mc{R}) \subseteq\msf{F}(\gamma_\mr{y}\mb{B}_\mc{Y},\gamma_\mr{r}\mb{B}_\mc{R}),
	\end{align*}
	with $\mb{B}_\bullet$ a unit ball of appropriate dimension. This can be solved directly through the iteration in Lemma~\ref{lem:disFcontr}. 
	On the other hand, if $\mc{Y},\mc{R}$ are restricted to hyper-rectangles, the operations to construct $\msf F$ simplify %
	number-based sum and difference operations. We will present more details on computing $\mc{Y},\mc{R}$ in future work.
\end{remark}}
Analysis and derivation of monotonicity/anti-tonicity properties of $\msf{F}$ w.r.t. individual design variables, such as $\Nr, \dny, K,A_\mr{p}$, etc., is still open and deferred to future work.

\section{Co-designing the LCA}\label{s:codesign}
The derivations in \S\ref{s:interactions} provided us with the error sources that dominate incorrectly designed LCAs and provided the map~$\msf{F}$. This map defines how the planner must be designed to guarantee constraint satisfaction of the LCA. We now consider these observations as $\phi,\theta$-dependent design constraints of the proposed LCA co-design problem. %
We first clarify the exact difference between the design variables~$\phi$ and~$\theta$ considered in the co-design problem. 
\subsection{LCA design variables}
Under the assumption that we are given $A,B,\mc{X},\mc{U}$ from~\eqref{eq:plant}, the choices to be made in the design for the considered two-layer control architecture can be divided up into three types of design variables: \emph{architecture-based} design variables~($\phi$), \emph{tuning-based} design variables~($\theta$), and \emph{algorithmic} design variables.
\begin{itemize}[align=left, leftmargin=*]
	\item[{\bf Architecture-based design variables}]
	influence the \emph{region} of stability and feasibility, and are about making tradeoffs in the design. In our setup, $\dny,\dnr,\Nr$, and $\tau_\mr{L}$ are evident architecture-based variables. As discussed in \S\ref{ss:plant+lca}, in this work we fix some architectural design choices, such as the type of controllers for~$\mbf C$ (state-feedback) and~$\mbf P$ (receding horizon predictive controller). The functional form of $\ell_\mr{goal},\ell_\mr{track}$, e.g., quadratic, linear, etc., is also an architectural design choice, while their \emph{parametrization} is part of the tuning-based design variables. 
	\item[{\bf Tuning-based design variables}]
	shape the behavior of the closed-loop LCA and ensure optimal operation of the architecture, given the design space defined by $\phi$. For this example, the numerical values for $\Pi_\mr{x}, K, A_\mr{p}, B_\mr{p}$, etc. are clearly tuning-based variables. %
	\item[{\bf Algorithmic design variables}]
	influence the tractability of the architecture design, and generally dictate the quality of finding~$\theta$. Examples in our setting are: the synthesis algorithm for~$K$, e.g., LQR, pole placement, or LMI-based; the reduction method from which~\eqref{eq:planmod} and~$\Pi_\bullet$ are obtained, e.g., balanced reduction, model matching, etc.; %
	or the discretization method of~\eqref{eq:ctplant}, e.g., exact, forward-Euler, RK4. While interesting and important, we do not consider the algorithmic component part of the co-design procedure in this paper, and defer this to future work.
\end{itemize}
For the setup we consider, we assume the following to concretely define $\phi$ and $\theta$, allowing us to formulate an LCA co-design problem that is \emph{actionable}.
\begin{assumption}\label{ass:costs}
	We are given $A_\mr{c},B_\mr{c},\mc{X},\mc{U}$,~\eqref{eq:ctplant} is discretized using exact discretization, and $\ell_\mr{goal},\ell_\mr{track}$ are defined as
	\begin{align*}
		\ell_\mr{goal}& := \textstyle \sum_{i=0}^{L-1}(\ast)^\top Q^\mbf{P}(\hat y_{i|n}-\msf{y}_\mr{goal}) \\ &\qquad+ (\ast)^\top R^\mbf{P}(r_{i|n}-\msf{r}_\mr{goal}) + (\ast)^\top Q_\mr{f}^\mbf{P}(\hat y_{L|n}-\msf{y}_\mr{goal}) \\
		\ell_\mr{track}&:=\textstyle\sum_{k=0}^\infty (\ast)^\top Q^\mbf{C}(\Pi_\mr{x}x_{k}-\bar{y}_n) + (\ast)^\top R^\mbf{C}(u_{k}-\bar{r}_n)
	\end{align*}
	where $\msf{y}_\mr{goal},\msf{r}_\mr{goal}$ are fixed a priori, and the matrices $Q^\mbf{P},Q_\mr{f}^\mbf{P},Q^\mbf{C},R^\mbf{P},R^\mbf{C}$ are tuning-based variables.
\end{assumption}
Define the architecture-based and tuning-based design variables~$\phi$ and~$\theta$ as the tuples:
\begin{subequations}
\begin{align}
	\phi&:=\{\dny,\dnr,\Nr, \tau_\mr{L}, L\}, \\
	\theta&:=\{K, A_\mr{p}, B_\mr{p}, \Pi_\mr{x}, \Pi_\mr{u},Q^\mbf{P},Q_\mr{f}^\mbf{P},Q^\mbf{C},R^\mbf{P},R^\mbf{C}\}.
\end{align}
\end{subequations}
We want to re-emphasize that Assumption~\ref{ass:costs} is simply a \emph{choice} we make to make the co-design problem concrete. As we will discuss in~\S\ref{s:discussion}, all the design-dependent aspects in Assumption~\ref{ass:costs} may be completely absorbed into $\phi$ and $\theta$.

With $\phi$ and $\theta$ defined, the map $\msf{F}$ in~\eqref{eq:fpit-constraints} can thus be written as a mapping dependent on the design variables $\phi$ and $\theta$, i.e.,~$\msf{F}_{\phi,\theta}(\mc{Y},\mc{R})$.
Moreover, define for some $\phi,\theta$ the functions $\msf{T}_\mbf{C}(\phi,\theta)$ and $\msf{T}_\mbf{P}(\phi,\theta)$ as representations for the maximum computation time required to compute $u_k$ and $(\bar{y}_n,\bar{r}_n)$  (for a pre-determined hardware platform), e.g., the number of flops required to compute a solution to~\eqref{eq:planner} for a given~$\phi,\theta$.

\subsection{Co-design problem formulation}
We now have all the ingredients to formulate the co-design problem enabling the joint design of $\mbf P$ and $\mbf C$, subject to the stability, compute and feasibility constraints. Combining all the above, the co-design problem for the considered two-layer LCA is as follows:
\begin{subequations}\label{eq:codesignproblem}
\begin{align}
	\min_{\phi,\theta} \quad & \mt{cost}(\phi,\theta) \label{eq:codesignproblem:cost}\\[-3pt] 
	\mr{s.t.}\quad & {\textstyle\rho(\Acl^{\Nr} - \Pi_\mr{x}^\top A_\mr{p} \Pi_\mr{x})<1, } \label{eq:codesignproblem:stab} \\
    & \exists(\mc{Y},\mc{R})\neq (\emptyset, \emptyset): (\mc{Y},\mc{R}) \subseteq \msf{F}_{\phi,\theta}(\mc{Y},\mc{R}),\label{eq:codesignproblem:feas}\\
	& \msf{T}_\mbf{C}(\phi,\theta)\leq \tau_\mr{L}, \quad \msf{T}_\mbf{P}(\phi,\theta)\leq\Nr\tau_\mr{L}. \label{eq:codesignproblem:comp}
\end{align}
\end{subequations}
The \emph{stability} constraint~\eqref{eq:codesignproblem:stab} ensures that the interactions between the layers in the closed-loop LCA are stable, and is directly derived from~\eqref{eq:errorevolve_alt:short}. The constraint~\eqref{eq:codesignproblem:feas} ensures that a particular choice of $\phi,\theta$ results in a \emph{feasible} LCA design: $\mbf{P}$ and $\mbf{C}$ must be such that by only actively enforcing (non-empty) constraints in $\mbf{P}$, $x_k\in\mc{X},u_k\in\mc{U}$ will be inherently satisfied by $\mbf{C}$. The constraint~\eqref{eq:codesignproblem:comp} is a \emph{computational} feasibility constraint, guaranteeing real-time implementation of the LCA: $(\bar{y}_n,\bar{r}_n)$ and $u_k$ must be computed within the sampling times $\tau_\mr{H}$ and $\tau_\mr{L}$, respectively. 

We demonstrated in the previous sections that~\eqref{eq:codesignproblem:stab}--\eqref{eq:codesignproblem:comp} are all computable and verifiable. Hence, making the co-design problem actionable and useful hinges on the definition of the cost function $\mt{cost}(\phi,\theta)$.

\subsection{Definition of $\mt{cost}$}\label{ss:cost}
$\mt{cost}(\phi,\theta)$ dictates between which axes of optimality the designer of the LCA can make informed tradeoffs. Based on user-defined specifications,~$\mt{cost}(\phi,\theta)$ maps choices of~$\phi,\theta$ to Pareto curves that represent a collection of optimal LCA designs. %
The different axes of $\mt{cost}$ simply favor/dominate different regions of the objective space, and the preference of the designer selects which region matters. 

The tradeoffs in LCA design are often heterogeneous and application specific, meaning that there generally does not exist a `unified' definition for~$\mt{cost}$. For the sake of illustration, we build up~$\mt{cost}$ along three axes: speed, accuracy, and compute, based on the classic tradeoff in control design. You cannot be fast, highly accurate, and have little computational effort: ``\emph{there is no free lunch}''. We compose $\mt{cost}$ with the scalar cost functions $J_\mr{speed}(\phi,\theta)$, $J_\mr{accuracy}(\phi,\theta)$, and $J_\mr{compute}(\phi,\theta)$. Systematic joint handling of these cost functions or scalarization of them into a manageable $\mt{cost}$ to explore the resulting Pareto tradeoffs is an active field of research, e.g.,~\cite{wilde2024scalarizing, alharbi2026compositional}. As this is not the focus of the paper, we just define it as a multi-valued cost function:
\[ \mt{cost}(\phi,\theta) := \begin{bsmallmatrix}J_\mr{speed}(\phi,\theta) & J_\mr{accuracy}(\phi,\theta) & J_\mr{compute}(\phi,\theta)\end{bsmallmatrix}^\top. \]
We give concrete examples of these cost components in \S\ref{s:example}.

\section{Case study}\label{s:example}
We solve the LCA co-design problem for a chain of coupled stable linear integrators. The plant is thus described by:
\begin{equation}\label{eq:examplesys}
	\dot{x} = \left[\begin{smallmatrix} 
	-a_0 & b_1 & b_2 & \cdots & \cdots & b_{\dnx-1}\\
	1 & -a_1 & 0 & \cdots & \cdots &0\\
	0 & 1 & \ddots & \ddots & \ddots & \vdots \\
	\vdots & \ddots & \ddots & \ddots & \ddots & \vdots \\
	\vdots & \ddots & \ddots & 1  & -a_{\dnx-2} & 0 \\
	0 & \cdots & \cdots & 0  & 1 & -a_{\dnx-1}
	\end{smallmatrix}\right] x + \left[\begin{smallmatrix} 
	1 \vphantom{b_{\dnx-1}}\\
	0\vphantom{b_{\dnx-1}}\\
	\vdots \vphantom{\ddots }\\
	\vdots  \vphantom{\ddots }\\
	0  \vphantom{a_{\dnx-2}}\\
	0 \vphantom{a_{\dnx-1}}
	\end{smallmatrix}\right] u,
\end{equation}
where $a_i,b_i>0$ and we define the constraint sets as $\mc{X}:=[-x_\mr{max},\, x_\mr{max}]^\dnx$ and $\mc{U}:=[-u_\mr{max},\, u_\mr{max}]^\dnu$. For illustrative purposes, we significantly simplify the co-design problem: we use exact discretization to obtain~\eqref{eq:plant} and exclude $\tau_\mr{L}$, $L$, $Q^\bullet$, and $R^\bullet$ from $\phi,\theta$. Moreover, we restrict $\Pi_\mr{x}$ to $\Pi_\mr{x} = [I_\mr{\dny}\, 0]$, and set $A_\mr{p}:= \Pi_\mr{x}A\Pi_\mr{x}^\top$, $B_\mr{p}:= \Pi_\mr{x}B$, $Q^\mbf{P} := \tfrac{1}{x_\mr{max}^2}I_{\dny}$, $Q^\mbf{C} := \tfrac{1}{x_\mr{max}^2}I_{\dnx}$, $Q^\mbf{P}_\mr{f} := 0$, $R^\mbf{P} := \tfrac{0.01}{u_\mr{max}^2}I_{\dnu}$, and $R^\mbf{C} := \tfrac{1}{u_\mr{max}^2}I_{\dnu}$. Finally, we optimize over~$K$ by solving the discrete-time algebraic Riccati equation, resulting in the optimal solution for a particular $\phi$. This reduces the co-design problem~\eqref{eq:codesignproblem} to a minimization problem over the (integer) architectural design variables~$\dny$ and~$\Nr$, which results in a tractable problem that traces out the Pareto-optimal LCA designs within this simplified setting.

To solve the LCA co-design problem, we will only need to define $J_\mr{speed}$, $J_\mr{accuracy}$, and $J_\mr{compute}$. These are design specific, cf.~\S\ref{ss:cost}, and we define them here from intuition.

\emph{Speed:} A measure for the `speed' of the LCA is $v_n:=\|\bar y_{n+1}-\bar y_n\|/\tau_\mr{H}$. Another measure for the speed of the LCA is $\Nr$ itself. Furthermore, if the tracker is aggressive and reaches $\bar{y}_n$ after only a fraction of $\Nr$, the two layers are not optimally working together. Therefore, we also penalize $|\tfrac{\mr{ln}(0.1)}{\mr{ln}(\rho(\Acl))}-\Nr|+\Nr=:\chi$, giving $J_\mr{speed}:= 1/v_n + \chi/x_\mr{max}$, where we normalized~$\chi$ using the range of~$x$.

\emph{Accuracy:} The measure of accuracy is expressed by the error between the (projected) state and the reference. We can evaluate this directly using $\mc{E}^\star$ in~\eqref{eq:mcEstar}, hence in $J_\mr{accuracy}$ we penalize $\sup_{e\in\mc{E}^\star(\mc{Y}^\star,\mc{R}^\star)}\|e\|$.

\emph{Compute:} In this case study, we assume that $\mbf{C}$ and $\mbf{P}$ \emph{share} their computational resource, meaning that $J_\mr{compute}$ penalizes $\msf{T}_\mbf{C}(\phi,\theta) + \tfrac{1}{\Nr}\msf{T}_\mbf{P}(\phi,\theta)$. Note that under this sharing assumption, we must modify~\eqref{eq:codesignproblem:comp} accordingly. As computing~\eqref{eq:controllaw} is straightforward matrix multiplication,~$\msf{T}_\mbf{C}\propto O(\dny\dnx)$, while computing a solution to the quadratic program~\eqref{eq:planner} scales with $O(n_\mr{y}^3 L)$. For illustration purposes, we use the proxy compute functions $\msf{T}_\mbf{C}:= C_1\dnx\dny$ and $\msf{T}_\mbf{P}:= C_2 n_\mr{y}^3 L$, where we take $C_1=C_2=5\mr{e}^{-6}$. %
\subsubsection*{Numerical results}
For our numerical results (see also {\small\tt tinyurl.com/lcacodesign}), we use the values below:
\begin{center}
{\footnotesize
	\begin{tabular}{cccccc}
		$\dnx$ & $\dnu$ & $\tau_\mr{L}$ & $x_\mr{max}$ & $u_\mr{max}$ & $L$ \\\hline
		20 & 1 & 0.02 & 10 & 4 & 10
	\end{tabular}
}
\end{center}
We solve the simplified co-design problem~\eqref{eq:codesignproblem} on the grid $\dny\times\Nr = [1,\,20]\times [2,\,30]$, and compute each axis of optimality for every grid point.
\begin{figure}
	\centering
	\includegraphics[width=\extendedwidth]{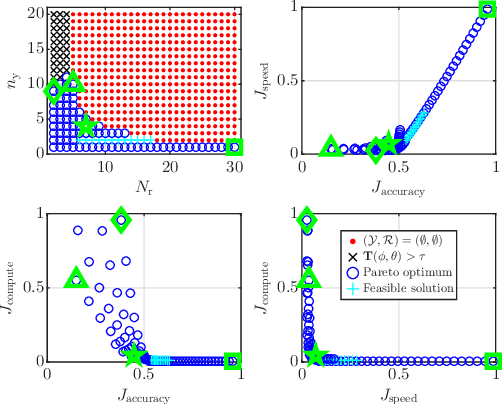}
	\caption{Top-left: Feasibility region of the LCA co-design problem~\eqref{eq:codesignproblem}, where %
	{\color{green}$\triangle$}, {\color{green}$\Diamond$}, {\color{green}$\square$}, {\color{green}$\bigstar$} correspond to a selection of LCA designs that are simulated in Fig.~\ref{fig:responses}. The top-right and bottom figures show the Pareto fronts along the different axes of optimality (normalized to 1) plotted against one another, allowing to make tradeoffs between speed, accuracy and compute.}\label{fig:paretos}
\end{figure}
This yields the plots in Fig.~\ref{fig:paretos}, where we have shown the feasible region on the grid $\dny\times\Nr$ in the top-left plot, and plotted the axes of optimality against each other in the remaining plots. The blue circles indicate Pareto-optimal solutions. Because we do not see $K$ as a tuning-based optimization variable in~\eqref{eq:codesignproblem}, the solutions of the simplified co-design problem on $\dny\times\Nr$ effectively trace out the Pareto curves directly. Note, however, that north of the shown Pareto curves there are (unshown) feasible solutions to~\eqref{eq:codesignproblem} with suboptimal/different numerical values for $K,A_\mr{p}, Q^\mbf{P}$, etc.

The Pareto curves demonstrate the evident tradeoff between speed, accuracy, and compute: If one desires cheap implementation, requiring a small $J_\mr{compute}$, the terms $J_\mr{speed}$ and $J_\mr{accuracy}$ grow. To show the implications of making certain tradeoffs, we selected four implementations on the Pareto curves: three focused on minimizing accuracy (\textcolor{green}{$\triangle$}), speed (\textcolor{green}{$\Diamond$}), or compute (\textcolor{green}{$\square$}), and one as the so-called diversity-enabled $\sigma$-sweet spot~\cite[p.~77]{matni2024quantitative} (\textcolor{green}{$\bigstar$}). The closed-loop responses of these LCA implementations are shown in Fig.~\ref{fig:responses}, where we track the first state of~\eqref{eq:examplesys} to 1. We verified that the simulations of the closed-loop LCA satisfy $x_k\in\mc{X},u_k\in\mc{U}$ (not shown).
\begin{figure}
	\centering
	\includegraphics[width=\extendedwidth]{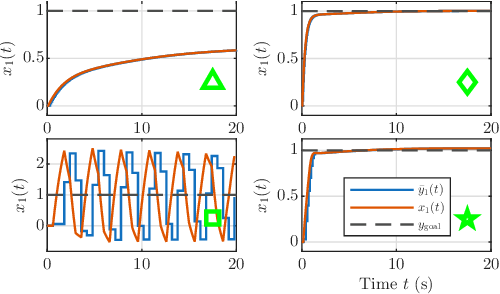}
	\caption{Responses of different Pareto-optimal LCA designs obtained from solving the co-design problem~\eqref{eq:codesignproblem}. The tradeoffs made in the LCA designs either leaned toward maximizing accuracy~(\textcolor{green}{$\triangle$}), speed~(\textcolor{green}{$\Diamond$}), minimizing computational load~(\textcolor{green}{$\square$}), or taking the diversity-enabled $\sigma$-sweet spot (\textcolor{green}{$\bigstar$}).}
	\label{fig:responses}
\end{figure}
The closed-loop responses clearly show the influence of the tradeoffs made for the individual designs of the LCAs. In the upper-left plot, speed is traded in for a high accuracy, i.e., a small error between $\bar{y}_n$ and~$x_{n\Nr+j}$ for all $n\in\mb{N}_{\geq0}$ and $j\in[0,\Nr-1]$, while in the lower-left plot, both speed and accuracy are traded for very cheap compute yielding only a marginally stable LCA. An informed tradeoff between the three considered components in~$\mt{cost}$ results in the lower-right response, which can be seen as the `consensus LCA' that is sufficiently accurate and fast, while meeting the compute limitations.

\section{Discussion and outlook}\label{s:discussion}
\emph{On solving~\eqref{eq:codesignproblem}:\,}
Although solving~\eqref{eq:codesignproblem} is nontrivial (large and heterogeneous solution space due to the definition of~$\phi,\theta$, nonlinear and non-convex constraints, a nonlinear and multi-objective cost), it only needs to be solved \emph{once offline}. If it takes a month to solve, we will know that the design we pick from the resulting Pareto front guarantees stability, constraint satisfaction, and is directly implementable. However, it is important that in further formalizing~\eqref{eq:codesignproblem} (see the following subsections), the co-design methods are polynomial-time solvable, requiring efficient ways to navigate the $\phi,\theta$-space.

\emph{On compositionality:\,}
For our simple two-layer LCA setup, we followed a particular path to arrive at~\eqref{eq:codesignproblem}: derive the interconnection dynamics (\S\ref{s:interactions}), formulate a constraint-tightening map (\S\ref{s:constraints}), and construct the multi-objective, joint optimization problem (\S\ref{s:codesign}). This path itself is not specifically tied to having exactly two layers. In fact, it opens up the possibility to follow it recursively: seeing the LCA in Fig.~\ref{fig:hierarchical} as the lower layer tracker as a whole, and having a higher layer decision maker above it. While conceptually straightforward, the actual form of the design constraints, corresponding to these new interfaces, is open: the resulting tracker is no longer the simple feedback law~\eqref{eq:controllaw}, but the closed-loop expression of the LCA in Fig.~\ref{fig:hierarchical}. Moreover, it is not obvious that an analogous~$\msf F$ exists with similar properties for such a composite system. Formalizing compositionality within our framework likely requires tools beyond the ones we presented. The category theory-based co-design framework of~\citet{zardini2023thesis, censi2015mathematical} seems a natural candidate to achieve this.

\emph{LCA co-design problem as a DMP:\,}
Problem~\eqref{eq:codesignproblem} closely resembles the single-layer DMP of~\citet{verhoek2026quantitative}: both optimize over design choices subject to stability and compute constraints. Making this resemblance precise, e.g., whether~\eqref{eq:codesignproblem} reduces to the single-layer DMP for a differently defined tuple~$\phi$ is worth pursuing. The connection would particularly be of interest when the control problem is highly complex and the individual layers of the baseline solution become expensive to evaluate in their own right (e.g., a learned world model or advanced motion planner). In that situation, the compute constraint may \emph{guide} the distribution of complexity and resources between layers, while guaranteeing stability and feasibility at the same time. Precisely characterizing the aforementioned resemblance in the context of the compositionality direction is, we believe, a promising horizon for the development of a general control architecture theory.

\emph{On noise and disturbances:\,}
The derivations in~\S\ref{s:interactions} and~\S\ref{s:constraints} naturally generalize for plants~\eqref{eq:plant} with a bounded additive disturbance~$v_k\in\mc{V}\subseteq\mb{R}^\dnx$. This disturbance term only expands the signal $g_n$, and requires the (straightforward) construction of an RPI set for the fast timescale as well. We expect that this directly leads to a \emph{robust} version of our LCA co-design problem with only minor modifications.

\section{Conclusions}\label{s:conclusion}
We presented a first step toward systematic and holistic design of \emph{layered control architectures}~(LCAs). We achieved this by building up a co-design problem to jointly optimize over the \emph{architecture-based} design variables, dictating the region of stability and feasibility of the LCA design, and the \emph{tuning-based} design variables, shaping the actual closed-loop behavior of the LCA. By considering a simple linear two-layer setup, we derived conditions, dependent on both design variables types, that guarantee stability, constraint satisfaction, and online implementability in terms of compute constraints of the closed-loop LCA. These conditions allow us to make systematic and informed tradeoffs on the closed-loop performance of the LCA in terms of speed, accuracy, and compute. Hence, we provide an explicit formulation to navigate the complex space of Pareto-optimal LCA designs.

\bibliographystyle{IEEEtran}
\bibliography{refs_acc27}

\arxivver{}{
\appendix
\section{Proofs and derivations}
Some of the results below make use of the tools presented in Appendix~\ref{app:aux}.
\subsection{Derivation of~\eqref{eq:errorevolve:simple}}\label{app:deriveslowerror}
We start with writing out the error at $k=(n+1)\Nr$:
\begin{align*}
	e_{0|n+1} & = x_{(n+1)\Nr} - {\Pi_\mr{x}^\top}\bar{y}_{n+1} =  x_{n\Nr+\Nr} - {\Pi_\mr{x}^\top}\bar{y}_{n+1}\\
	& \overset{\text{\eqref{eq:plannerupdaterelation}}}{=} \Acl x_{n\Nr+\Nr-1} - BK\bar{y}_n + B \bar{r}_n - \Pi_\mr{x}^\top \hat y_{1|n} \\
	& = \Acl x_{n\Nr+\Nr-1} - BK\bar{y}_n + B \bar{r}_n \\ & \hspace{2cm} - \Pi_\mr{x}^\top A_\mr{p} \hat y_{0|n} - \Pi_\mr{x}^\top B_\mr{p} r_{0|n} \\
	& \overset{\text{\eqref{eq:planner:stateupdate}}}{=} \Acl x_{n\Nr+\Nr-1} - BK\bar{y}_n + B \bar{r}_n \\ & \hspace{2cm} - \Pi_\mr{x}^\top A_\mr{p} \Pi_\mr{x}x_{n\Nr} - \Pi_\mr{x}^\top B_\mr{p} r_{0|n}
\end{align*}
Recall that~$S_\Nr = \sum_{i=0}^{\Nr-1}\Acl^i$. Then, substituting~\eqref{eq:errorevolution} for $i=\Nr-1$ and observing that $\Acl{\Pi_\mr{x}^\top}\bar{y}_n - BK\bar{y}_n = A{\Pi_\mr{x}^\top}\bar{y}_n$ yields
\begin{align*}
	e_{0|n+1} & = \Acl^{\Nr} e_{0|n} + \left({\textstyle \sum_{j=1}^{\Nr-1}\Acl^j}\right)g_n + A{\Pi_\mr{x}^\top}\bar{y}_n \\ & \qquad + B \bar{r}_n - \Pi_\mr{x}^\top A_\mr{p} \Pi_\mr{x}x_{n\Nr} - \Pi_\mr{x}^\top B_\mr{p} r_{0|n}\\
	& = \Acl^{\Nr} e_{0|n} + S_\Nr g_n + {\Pi_\mr{x}^\top}\bar{y}_n \\ & \qquad - \Pi_\mr{x}^\top A_\mr{p} \Pi_\mr{x}x_{n\Nr} - \Pi_\mr{x}^\top B_\mr{p} r_{0|n}.
\end{align*}
Finally, by recognizing that the interaction relations~\eqref{eq:planner:stateupdate} and~\eqref{eq:plannerupdaterelation} dictate ${\Pi_\mr{x}^\top}\bar{y}_{n+1} = \Pi_\mr{x}^\top\hat{y}_{1|n} = \Pi_\mr{x}^\top A_\mr{p}\hat{y}_{0|n} + \Pi_\mr{x}^\top B_\mr{p} r_{0|n} = \Pi_\mr{x}^\top A_\mr{p} \Pi_\mr{x}x_{n\Nr}+ \Pi_\mr{x}^\top B_\mr{p} r_{0|n}$, we obtain~\eqref{eq:errorevolve:simple}.

\subsection{Proof of Lemma~\ref{lem:antitoneF}}
Let $\mc{Y}_1\subseteq\mc{Y}_2$ and $\mc{R}_1\subseteq\mc{R}_2$ and fix $\phi,\theta$. First note that, trivially, $\bar{a}(\mc{Y}_1)\leq\bar{a}(\mc{Y}_2)$, $\bar{b}(\mc{R}_1)\leq\bar{b}(\mc{R}_2)$, and $\bar{g}(\mc{Y}_1, \mc{R}_1)\leq\bar{g}(\mc{Y}_2, \mc{R}_2)$. From Fact~\ref{fact:sums}, we immediately get that $\mc{W}(\mc{Y}_1, \mc{R}_1)\subseteq\mc{W}(\mc{Y}_2, \mc{R}_2)$ and $\mc{G}(\mc{Y}_1, \mc{R}_1)\subseteq\mc{G}(\mc{Y}_2, \mc{R}_2)$. Under $\rho(\Acl^{\Nr} - \Pi_\mr{x}^\top A_\mr{p} \Pi_\mr{x})<1$, for any $t\in\mb{N}$, $(\Acl^{\Nr} - \Pi_\mr{x}^\top A_\mr{p} \Pi_\mr{x})^t$ is finite. Consequently, Fact~\ref{fact:mult} yields that $(\Acl^{\Nr} - \Pi_\mr{x}^\top A_\mr{p} \Pi_\mr{x})^t\mc{W}(\mc{Y}_1,\mc{R}_1)\subseteq(\Acl^{\Nr} - \Pi_\mr{x}^\top A_\mr{p} \Pi_\mr{x})^t\mc{W}(\mc{Y}_2,\mc{R}_2)$. Repeated application of Fact~\ref{fact:sums} gives that $\mc{E}_0(\mc{Y}_1,\mc{R}_1)\subseteq\mc{E}_0(\mc{Y}_2,\mc{R}_2)$. Using the same arguments, $\mc{E}_j(\mc{Y}_1,\mc{R}_1)\subseteq\mc{E}_j(\mc{Y}_2,\mc{R}_2)$, and since set inclusion of monotone maps is preserved for the union operator, $\mc{E}^\star(\mc{Y}_1,\mc{R}_1)\subseteq\mc{E}^\star(\mc{Y}_2,\mc{R}_2)$. Finally, we use Fact~\ref{fact:diff} and the definition of $\msf F$ in~\eqref{eq:fpit-constraints} to conclude that for $\mc{Y}_1\subseteq\mc{Y}_2$ and $\mc{R}_1\subseteq\mc{R}_2$, we have that $\msf{F}(\mc{Y}_2,\mc{R}_2)\subseteq\msf{F}(\mc{Y}_1,\mc{R}_1)$, i.e., for fixed design variables $\phi,\theta$, the map $\msf{F}_{\phi,\theta}$ is order-reversing w.r.t. $(\mc{Y},\mc{R})$. That is, $\msf{F}_{\phi,\theta}$ is anti-tone in $(\mc{Y},\mc{R})$, concluding the proof. \qed

\subsection{Proof of Lemma~\ref{lem:disFcontr}}
	It is trivial to see that if we construct $\mc{E}^\star(\mc{S}_0)$ by following the procedure in \S\ref{ss:construction-of-Y-R},~$0\in\mc{E}^\star(\mc{S}_0)$. Therefore, $\mc{S}_1 = \msf{F}_{\phi,\theta}(\mc{S}_0)\subseteq\mc{S}_0$, i.e.,~$\underline{\mc{S}}_0\subseteq\overline{\mc{S}}_0$. Let $\msf{F}^2_{\phi,\theta}:=\msf{F}_{\phi,\theta}\circ\msf{F}_{\phi,\theta}$, and note that $\underline{\mc{S}}_{k+1} = \msf{F}^2_{\phi,\theta}(\underline{\mc{S}}_k)$ and $\overline{\mc{S}}_{k+1} = \msf{F}^2_{\phi,\theta}(\overline{\mc{S}}_k)$. From Lemma~\ref{lem:antitoneF}, we can directly conclude that $\msf{F}^2_{\phi,\theta}$ is monotone, i.e., order-preserving. Repeatedly applying $\msf{F}^2_{\phi,\theta}$ on the inclusion relation~$\underline{\mc{S}}_0\subseteq\overline{\mc{S}}_0$ gives
	\begin{equation}\label{eq:pflm2:1}
		\underline{\mc{S}}_k\subseteq\overline{\mc{S}}_k.
	\end{equation}
	Using the same argument as before, we conclude that for \emph{any} $\mc{S}$, $\msf{F}_{\phi,\theta}(\mc{S})\subseteq\mc{S}_0$, yielding $\overline{\mc{S}}_1=\mc{S}_2\subseteq\mc{S}_0=\overline{\mc{S}}_0$. Applying $\msf{F}_{\phi,\theta}$ once on the latter inclusion reverses the order and thus yields~$\underline{\mc{S}}_0=\mc{S}_1\subseteq\mc{S}_3=\underline{\mc{S}}_1$. We now repeatedly apply $\msf{F}^2_{\phi,\theta}$ on these two inclusion relationships to arrive at
	\begin{equation}\label{eq:pflm2:2}
		\overline{\mc{S}}_{k+1}\subseteq\overline{\mc{S}}_k, \quad \text{and} \quad \underline{\mc{S}}_k\subseteq\underline{\mc{S}}_{k+1}.
	\end{equation}
	Combining~\eqref{eq:pflm2:2} with~\eqref{eq:pflm2:1} gives the inclusion chain $\underline{\mc{S}}_k\subseteq\underline{\mc{S}}_{k+1}\subseteq\overline{\mc{S}}_{k+1}\subseteq\overline{\mc{S}}_k$. Applying Lemma~\ref{lem:hausdorffineq} on this inclusion chain gives the desired result, i.e., $\mc{D}_{k+1} = d_H(\underline{\mc{S}}_{k+1}, \overline{\mc{S}}_{k+1})\leq d_H(\underline{\mc{S}}_{k}, \overline{\mc{S}}_{k}) =  \mc{D}_{k}$.
\qed

\section{Auxiliary results}\label{app:aux}
\subsection{Facts regarding $\oplus,\ominus$}
\begin{fact}\label{fact:sums}
	Let $A,B,C,D\subseteq\mb{R}^n$. If $A\subseteq B$ and $C\subseteq D$, then $A\oplus C \subseteq B\oplus D$.
\end{fact}
\begin{fact}\label{fact:mult}
	For $M\in\mb{R}^{m\times n}$ and $A,B\subseteq\mb{R}^n$, let $MA:=\{ y \mid y = Mx, x\in A\}$. If $A\subseteq B$, then $MA\subseteq MB$.
\end{fact}
\begin{fact}\label{fact:diff}
	Let $A,B,C\subseteq\mb{R}^n$. If $A\subseteq B$, then $C\ominus B \subseteq C\ominus A$.
\end{fact}
The proofs for these facts directly follow from the definition of the Minkowski sum and Pontryagin difference.
\subsection{Hausdorff distances and set inclusions}\label{app:hausdorff}
Let $(M,h)$ be a metric space. Let $d$ be a distance function, quantifying the distance from a point $a\in X\subset M$ to the subset $B\subseteq X\subset M$, which is defined as \(d(a,B):=\inf_{b\in B} h(a,b)\). So,~the function~$d$ represents the distance between $a$ and the point in $B$ that is closest to $a$. For $A,B\subset M$, $\sup_{a\in A} d(a,B)$ thus describes the largest closest distance from any point in $A$ to $B$. The Hausdorff distance symmetrizes this, i.e., it defines a measure for how far apart two subsets of a metric space are from each other.
\begin{definition}[Hausdorff distance]
	Let $(M,h)$ be a metric space and consider non-empty subspaces $A,B\subset M$. The Hausdorff distance $d_H$ between $A$ and $B$ is defined as
	\[ d_H(A,B):= \max\left\{{\textstyle \sup_{a\in A} d(a,B), \ \sup_{b\in B} d(A,b) }\right\}. \]
\end{definition}
Thus, the Hausdorff distance is the largest possible distance from any point in a set to the closest possible point in another set.
This allows to formulate the following lemma for $A_1,A_2,B_1,B_2\subset M$:
\begin{lemma}\label{lem:hausdorffineq}
	If $B_1\subseteq B_2 \subseteq A_2 \subseteq A_1$, then $d_H(A_2,B_2)\leq d_H(A_1,B_1)$.
\end{lemma}
\begin{proof}
	Since $B_i\subseteq A_i$, $\sup_{b\in B_i}d(A_i,b) = 0$ and thus $d_H(A_i,B_i) = \sup_{a\in A_i}d(a,B_i)$. Moreover, as $A_2\subseteq A_1$, we have that $d_H(A_2,B_2) = \sup_{a\in A_2}d(a,B_2) \leq \sup_{a\in A_1}d(a,B_2)$, because we consider the supremum over a larger domain. On the other hand, since $B_1\subseteq B_2$, $\sup_{a\in A_1}d(a,B_2)\leq \sup_{a\in A_1}d(a,B_1) = d_H(A_1,B_1)$, as the set we compare against is smaller. Combining these, we get $d_H(A_2,B_2)\leq d_H(A_1,B_1)$, proving the claim. 
\end{proof}

}

\end{document}